\documentclass[12pt,twoside]{amsart}
\usepackage{amssymb,amsmath,amsthm, amscd, enumerate, mathrsfs}
\usepackage{graphicx, hhline}
\usepackage[all]{xy}

\title[Fundamental properties of 
basic slc-trivial fibrations, II]
{Fundamental properties of basic slc-trivial fibrations, II} 
\author{Osamu Fujino, Taro Fujisawa, and Haidong Liu}
\date{2020/4/21, version 0.28}
\address{Department of Mathematics, Graduate School of Science, 
Osaka University, Toyonaka, Osaka 560-0043, Japan}
\email{fujino@math.sci.osaka-u.ac.jp}
\address{Department of Mathematics, School of Engineering, 
Tokyo Denki University, Tokyo, Japan}
\email{fujisawa@mail.dendai.ac.jp}
\address{Peking University, Beijing International 
Center for Mathematical Research,
Beijing, 100871, China}
\email{hdliu@bicmr.pku.edu.cn}

\DeclareMathOperator{\Gr}{Gr}
\DeclareMathOperator{\Supp}{Supp}
\DeclareMathOperator{\Spec}{Spec}
\DeclareMathOperator{\mult}{mult}
\DeclareMathOperator{\rank}{rank}
\newtheorem{thm}{Theorem}[section]
\newtheorem{lem}[thm]{Lemma}

\newtheorem{conj}[thm]{Conjecture}
\newtheorem{cor}[thm]{Corollary}

\theoremstyle{definition}

\newtheorem{defn}[thm]{Definition}
\newtheorem{rem}[thm]{Remark}
\newtheorem*{ack}{Acknowledgments}  
\newtheorem*{conventions}{Conventions}        
\newtheorem{say}[thm]{}
\newtheorem{step}{Step}

\makeatletter
    
    \@addtoreset{equation}{section}
\makeatother

\begin{document}

\begin{abstract}
We prove that if the moduli $\mathbb Q$-b-divisor 
of a basic slc-trivial fibration 
is b-numerically trivial then 
it is $\mathbb Q$-b-linearly trivial. As a consequence, we prove that 
the moduli part of a basic slc-trivial fibration is semi-ample when 
the base space is a curve. 
\end{abstract}

\maketitle 
\tableofcontents

\section{Introduction}\label{B-sec1}

This paper is a continuation of the first author's paper:~\cite{fujino-slc-trivial}. 
We strongly recommend the reader to look 
at \cite[1.~Introduction]{fujino-slc-trivial} before starting to read this paper. 
In \cite{fujino-slc-trivial}, we 
introduced the notion of basic slc-trivial fibrations, which 
is a kind of canonical bundle formula for reducible varieties, and 
investigated some fundamental properties. 
For the precise definition of basic slc-trivial fibrations, 
see \cite[Definition 4.1]{fujino-slc-trivial} or 
Definition \ref{C-def3.1} below. 
The following statement is one of the main results of \cite{fujino-slc-trivial}. 

\begin{thm}[{\cite[Theorem 1.2]{fujino-slc-trivial}}]\label{B-thm1.1}
Let $f\colon (X, B)\to Y$ be a basic slc-trivial fibration and 
let $\mathbf B$ and $\mathbf M$ be the induced 
discriminant and moduli $\mathbb Q$-b-divisors of $Y$ respectively. 
Then we have the following properties: 
\begin{itemize}
\item[(i)] $\mathbf K+\mathbf B$ is $\mathbb Q$-b-Cartier, and 
\item[(ii)] $\mathbf M$ is b-potentially nef, that is,  
there exists a proper birational morphism $\sigma \colon Y'\to Y$ 
from a normal variety $Y'$ such that 
$\mathbf M_{Y'}$ is a potentially nef $\mathbb Q$-divisor on $Y'$ and 
that $\mathbf M=\overline{\mathbf M_{Y'}}$. 
\end{itemize}
\end{thm}

For the definition and some basic properties of b-divisors, 
see \cite[2.3.2 b-divisors]{corti} and \cite[Section 2]{fujino-slc-trivial}. 

\medskip 

On moduli $\mathbb Q$-b-divisors, we have the following conjecture, 
which is still widely open. 

\begin{conj}[{b-semi-ampleness conjecture, 
see \cite[Conjecture 1.4]{fujino-slc-trivial}}]
\label{B-conj1.2} 
Let $f \colon (X, B)\to Y$ be a basic 
slc-trivial fibration. 
Then the moduli $\mathbb Q$-b-divisor $\mathbf M$ is 
b-semi-ample. 
\end{conj}

The main purpose of this paper is to prove the following theorem. 

\begin{thm}[Main Theorem]\label{B-thm1.3}
Let $f \colon (X, B)\to Y$ be a basic slc-trivial fibration 
such that $Y$ is complete. 
Let $\mathbf M$ be the moduli $\mathbb Q$-b-divisor 
associated to $f \colon (X, B)\to Y$. 
Assume that 
there exists a proper birational 
morphism $\sigma \colon Y'\to Y$ from a normal 
variety $Y'$ such that 
$\mathbf M=\overline {\mathbf M_{Y'}}$ with $\mathbf M_{Y'}\equiv 0$. Then 
$\mathbf M_{Y'}\sim _{\mathbb Q}0$ holds. 
\end{thm}

Theorem \ref{B-thm1.3} solves Conjecture \ref{B-conj1.2} 
when the moduli $\mathbb Q$-b-divisor $\mathbf M$ is 
b-numerically trivial.
It is obviously a generalization of \cite[Theorem 3.5]{ambro-moduli} and 
\cite[Theorem 1.3]{floris}. 
More precisely, Florin Ambro and Enrica Floris 
proved Theorem \ref{B-thm1.3} for klt-trivial fibrations and lc-trivial fibrations, 
respectively. 

\medskip

As a direct consequence of Theorem \ref{B-thm1.3}, 
we have the following result:~Corollary \ref{B-cor1.4}. 
It says that the b-semi-ampleness 
conjecture 
(see Conjecture \ref{B-conj1.2}) holds true when the base space is a curve. 
Note that Corollary \ref{B-cor1.4} was already proved for klt-trivial fibrations 
by Florin Ambro (see \cite[Theorem 0.1]{ambro-shokurov}). 

\begin{cor}\label{B-cor1.4}
Let $f \colon (X, B)\to Y$ be a basic slc-trivial fibration 
with $\dim Y=1$. 
Then the moduli $\mathbb Q$-divisor $M_Y$ of $f \colon 
(X, B)\to Y$ is 
semi-ample. 
\end{cor}

For the proof of Theorem \ref{B-thm1.3}, we closely follow Floris's arguments in 
\cite{floris}. We adapt her proof of Theorem \ref{B-thm1.3} for lc-trivial fibrations 
to our setting. As is well known, the 
main ingredient of \cite[Theorem 0.1]{ambro-shokurov}, 
\cite[Theorem 3.5]{ambro-moduli}, and \cite[Theorem 1.3]{floris} 
is Deligne's result on local subsystems of polarizable variations 
of $\mathbb Q$-Hodge structure (see \cite[Corollaire (4.2.8)]{deligne}). 

\medskip 

In \cite{fujino-fujisawa}, the first and the second 
authors discussed variations of mixed 
Hodge structure toward applications for 
higher-dimensional algebraic varieties (see also \cite{ffs}). 
One of the most important 
applications of \cite{fujino-fujisawa} is the proof of the 
projectivity of the coarse moduli spaces of stable 
varieties in \cite{fujino-ann}. Then the first author 
introduced the notion of basic slc-trivial fibrations 
in \cite{fujino-slc-trivial} in order to make results 
in \cite{fujino-fujisawa} useful 
for various geometric applications. The 
first and the third authors established that 
every quasi-log canonical pairs have only 
Du Bois singularities in \cite{fujino-haidong} by using 
\cite{fujino-slc-trivial}. 
We strongly recommend the reader to look 
at \cite[1.~Introduction]{fujino-slc-trivial} for 
more details. In this paper, we prove 
\cite[Conjecture 1.4]{fujino-slc-trivial} under 
some special assumption. We freely use the 
formulation introduced in \cite{fujino-slc-trivial} 
and the 
arguments in this paper heavily depend on \cite{fujino-fujisawa}.

\medskip 

We briefly explain the organization of this paper. 
In Section \ref{A-sec2}, 
we fix the notation and recall some definitions for the reader's convenience. 
In Section \ref{C-sec3}, we quickly recall the notion of basic slc-trivial fibrations 
and some definitions following \cite{fujino-slc-trivial}. 
In Section \ref{d-sec4}, we see that 
the cyclic group action constructed in 
\cite[Section 6]{fujino-slc-trivial} preserves some parts of weight filtrations of 
the variation of mixed Hodge structure. 
Section \ref{C-sec5} is devoted to the proof of Theorem \ref{B-thm1.3}. 
By using the result obtained in Section \ref{d-sec4}, 
we reduce Theorem \ref{B-thm1.3} to 
Deligne's result on local subsystems of polarizable variations 
of $\mathbb Q$-Hodge structure. 

\begin{ack}
The first author was partially 
supported by JSPS KAKENHI Grant Numbers 
JP16H03925, JP16H06337. 
The second author was partially supported by JSPS KAKENHI 
Grant Number JP16K05107. 
The authors would like to thank Takeshi Abe for 
useful discussions and comments. 
\end{ack}

\begin{conventions} 
We work over $\mathbb C$, the complex number field, throughout 
this paper. We freely use the basic 
notation of the minimal model program as in 
\cite{fujino-fundamental} and \cite{fujino-foundations}. 
A {\em{scheme}} means a separated scheme of 
finite type over $\mathbb C$. 
A {\em{variety}} means a reduced scheme, that is, 
a reduced separated scheme of finite type over $\mathbb C$. 
In this paper, a variety may be reducible. 
However, we sometimes assume that a variety is irreducible without 
mentioning it explicitly if there is no danger of confusion. 
The set of integers (resp.~rational numbers) 
is denoted by $\mathbb Z$ (resp.~$\mathbb Q$). 
The set of positive rational numbers (resp.~integers) 
is denoted by $\mathbb Q_{>0}$ (resp.~$\mathbb Z_{>0}$). 
\end{conventions}

In this paper, we do not use $\mathbb R$-divisors. 
We only use $\mathbb Q$-divisors. 

\section{Preliminaries}\label{A-sec2}
In this section, we quickly recall some basic definitions and notation 
for the reader's convenience. 
For the details, see \cite[Section 2]{fujino-slc-trivial}. 

\medskip 

Let us start with the definition of {\em{simple normal crossing pairs}}. 

\begin{defn}[Simple normal crossing pairs]\label{A-def2.1}
We say that the pair $(X, B)$ is {\em{simple normal crossing}} 
at a point $a\in X$ if $X$ has a Zariski open neighborhood 
$U$ of $a$ that can be embedded in a smooth 
variety $M$, where $M$ has a regular system of parameters 
$(x_1, \ldots, x_p, y_1, \ldots, y_r)$ at $a=0$ in 
which $U$ is defined by a monomial equation 
$$
x_1\cdots x_p=0
$$ 
and 
$$
B=\sum _{i=1}^r b_i (y_i=0)|_U, \quad 
b_i\in \mathbb Q.
$$ 
We say that $(X, B)$ is a {\em{simple normal crossing pair}} 
if it is simple normal crossing at every point of $X$. 
If $(X, 0)$ is a simple normal crossing pair, then 
$X$ is called a {\em{simple normal crossing variety}}. 
If $(X, B)$ is a simple normal crossing pair and 
$B$ is reduced, then $B$ is called a {\em{simple 
normal crossing divisor}} on $X$. 

Let $(X, B)$ be a simple normal crossing pair 
such that all the coefficients of $B$ are 
less than or equal to one. 
Let $\nu \colon X^\nu\to X$ be the normalization of $X$. 
We put $K_{X^\nu}+\Theta=\nu^*(K_X+B)$, that is, 
$\Theta$ is the sum of 
the inverse images of $B$ and the singular locus of $X$. 
By assumption, all the coefficients of $\Theta$ are less than or equal to one. 
Therefore, it is easy to see that $(X^\nu, \Theta)$ is sub log canonical. 
In this situation, we simply say that $W$ is 
a {\em{stratum}} of $(X, B)$ if $W$ is an irreducible component of $X$ or 
$W$ is the $\nu$-image of some log canonical center of $(X^\nu, \Theta)$. 
We note that a stratum of a simple normal crossing variety $X$ 
means a stratum of a simple normal crossing pair 
$(X, 0)$. 
\end{defn}

We write the precise definition of {\em{semi-log canonical pairs}}, 
{\em{slc centers}}, and {\em{slc strata}} for the reader's convenience. 
For the details of semi-log canonical pairs, we recommend the reader to 
see \cite{fujino-fund-slc}. 

\begin{defn}[Semi-log canonical pairs]\label{x-def2.2} 
Let $X$ be an equidimensional scheme which 
satisfies Serre's $S_2$ condition and 
is normal crossing in codimension one. Let $\Delta$ 
be an effective $\mathbb Q$-divisor on $X$ 
such that no irreducible component of $\Supp \Delta$ 
is contained in the singular locus of $X$ and that 
$K_X+\Delta$ is $\mathbb Q$-Cartier. 
We say that $(X, \Delta)$ is a {\em{semi-log canonical}} pair if 
$(X^\nu, \Delta_{X^\nu})$ is log canonical 
in the usual sense, where $\nu:X^\nu\to X$ 
is the normalization of $X$ and 
$K_{X^\nu}+\Delta_{X^\nu}=\nu^*(K_X+\Delta)$, 
that is, $\Delta_{X^\nu}$ is the sum of the inverse 
images of $\Delta$ 
and the conductor of $X$. An {\em{slc center}} of $(X, \Delta)$ 
is the $\nu$-image of an lc center of $(X^\nu, \Delta_{X^\nu})$. 
An {\em{slc stratum}} of $(X, \Delta)$ 
means either an slc center of $(X, \Delta)$ or an 
irreducible component of $X$. 
\end{defn}

We recall various definitions and operations of 
($\mathbb Q$-)divisors. 
We note that we are mainly interested in {\em{reducible}} varieties 
in this paper. 

\begin{say}[Divisors]\label{x-say2.3} 
Let $X$ be a scheme with structure sheaf $\mathcal O_X$ and let 
$\mathcal K_X$ be the sheaf of total quotient rings of $\mathcal O_X$. 
Let $\mathcal K^*_X$ denote the (multiplicative) 
sheaf of invertible elements in $\mathcal K_X$, 
and $\mathcal O^*_X$ the sheaf of invertible 
elements in $\mathcal O_X$. 
We note that $\mathcal O_X\subset \mathcal K_X$ 
and $\mathcal O^*_X\subset 
\mathcal K^*_X$ hold. 
A {\em{Cartier divisor}} $D$ on $X$ is a global section of 
$\mathcal K^*_X/\mathcal O^*_X$, that is, 
$D$ is an element of $\Gamma(X, \mathcal K^*_X/\mathcal O^*_X)$. 
A {\em{$\mathbb Q$-Cartier divisor}} is an element of 
$\Gamma (X, \mathcal K^*_X/\mathcal O^*_X)\otimes 
_{\mathbb Z}\mathbb Q$. Let $D_1$ and $D_2$ 
be two $\mathbb Q$-Cartier divisors 
on $X$. Then $D_1$ is {\em{linearly}} 
(resp.~{\em{$\mathbb Q$-linearly}}) 
{\em{equivalent}} to $D_2$, denoted by $D_1\sim D_2$ (resp.~$D_1
\sim _{\mathbb Q}D_2$), if 
$$
D_1=D_2+\sum _{i=1}^k r_i (f_i)
$$ 
such that $f_i\in \Gamma (X, \mathcal K^*_X)$ and $r_i\in \mathbb Z$ 
(resp.~$r_i\in \mathbb Q$) for every $i$. 
We note that $(f_i)$ is a {\em{principal Cartier divisor}} 
associated to $f_i$, that is, 
the image of $f_i$ by 
$$
\Gamma (X, \mathcal K^*_X)\to 
\Gamma(X, \mathcal K^*_X/\mathcal O^*_X). 
$$
Let $f \colon X\to Y$ be a morphism between schemes. 
If there exists a $\mathbb Q$-Cartier 
divisor $B$ on $Y$ such that 
$D_1\sim _{\mathbb Q} D_2+f^*B$, then 
$D_1$ is said to be {\em{relatively $\mathbb Q$-linearly 
equivalent to $D_2$}}. 
It is denoted by $D_1\sim _{\mathbb Q, f}D_2$ or 
$D_1\sim _{\mathbb Q, Y} D_2$. 

\medskip

From now on, let $X$ be an equidimensional scheme. We note 
that $X$ is not necessarily regular in codimension one. 
A ({\em{Weil}}) {\em{divisor}} $D$ on $X$ is a finite formal 
sum 
$$
D=\sum _i d_iD_i, 
$$
where $D_i$ is an irreducible reduced closed subscheme of $X$ 
of pure codimension one and $d_i$ is an integer 
for every $i$ such that $D_i\ne D_j$ for every $i\ne j$. 
If $d_i \in \mathbb Q$ for every $i$, 
then $D$ is called a {\em{$\mathbb Q$-divisor}}. 
Let $D=\sum _i d_i D_i$ be a $\mathbb Q$-divisor as above. 
We put 
\begin{equation*}
D^{\leq 1}=\sum _{d_i\leq 1}d_i D_i, \quad 
D^{<1} =\sum _{d_i<1}d_iD_i, \quad 
D^{= 1}=\sum _{d_i= 1} D_i, \quad \text{and} \quad
\lceil D\rceil =\sum _i \lceil d_i \rceil D_i, 
\end{equation*}
where $\lceil d_i\rceil$ is the integer defined by $d_i\leq 
\lceil d_i\rceil <d_i+1$. Let $D$ be a $\mathbb Q$-divisor. 
We also put 
$$
\lfloor D\rfloor=-\lceil -D\rceil. 
$$
We call $D$ a {\em{subboundary}} 
$\mathbb Q$-divisor if $D=D^{\leq 1}$ holds. 
When $D$ is effective and $D=D^{\leq 1}$ holds, 
we call $D$ a {\em{boundary}} $\mathbb Q$-divisor. 

We further assume that 
$f \colon X\to Y$ is a surjective morphism onto an irreducible 
variety $Y$. 
Then we put 
$$
D^v=\sum _{f(D_i)\subsetneq Y}d_i D_i \quad 
\text{and} \quad D^h=D-D^v, 
$$
and call $D^v$ the {\em{vertical part}} 
and $D^h$ the {\em{horizontal part}} of $D$ 
with respect to $f \colon X\to Y$, respectively. 

\medskip 

Finally, let $D$ be a $\mathbb Q$-Cartier divisor on a 
complete normal irreducible variety $X$. 
If $D\cdot C=0$ for any complete curve $C$ on $X$, then 
$D$ is said to be {\em{numerically trivial}}. When $D$ is 
numerically trivial, we simply write $D\equiv 0$. 
\end{say}

Let us recall the definition of {\em{potentially nef divisors}} 
introduced by the first author in \cite{fujino-slc-trivial}. 

\begin{defn}[{Potentially nef divisors, see 
\cite[Definition 2.5]{fujino-slc-trivial}}]\label{x-def2.4} 
Let $X$ be a normal 
irreducible variety and let $D$ be a divisor on $X$. 
If there exist a completion $X^\dag$ of $X$, 
that is, $X^\dag$ is a complete normal 
variety and contains $X$ as a dense Zariski open set, and 
a nef divisor $D^\dag$ on $X^\dag$ such that 
$D=D^\dag|_X$, then $D$ is called 
a {\em{potentially nef}} divisor on $X$. 
A finite $\mathbb Q_{>0}$-linear 
combination of potentially nef divisors is called 
a {\em{potentially nef}} $\mathbb Q$-divisor. 
\end{defn}

Although it is dispensable, 
the following definition is very useful when we state our results (see 
Theorems \ref{B-thm1.1} and \ref{B-thm1.3}). 
We note that the {\em{$\mathbb Q$-Cartier 
closure}} of a $\mathbb Q$-Cartier $\mathbb Q$-divisor 
$D$ on a normal variety $X$ is the $\mathbb Q$-b-divisor 
$\overline D$ with trace 
$$
\overline D _Y=f^*D, 
$$
where $f \colon Y\to X$ is a proper birational morphism 
from a normal variety $Y$. 

\begin{defn}[{see \cite[Definition 2.12]{fujino-slc-trivial}}]\label{x-def2.5}
Let $X$ be a normal irreducible variety. 
A $\mathbb Q$-b-divisor $\mathbf D$ of $X$ 
is {\em{b-potentially nef}} 
(resp.~{\em{b-semi-ample}}) if there 
exists a proper birational morphism $X'\to X$ from a normal 
variety $X'$ such that $\mathbf D=\overline {\mathbf D_{X'}}$, that 
is, $\mathbf D$ is the $\mathbb Q$-Cartier closure of $\mathbf D_{X'}$, and that 
$\mathbf D_{X'}$ is potentially nef 
(resp.~semi-ample). 
A $\mathbb Q$-b-divisor $\mathbf D$ of $X$ is {\em{$\mathbb Q$-b-Cartier}} 
if there is a proper birational morphism $X'\to X$ from a normal 
variety $X'$ such that $\mathbf D=\overline{\mathbf D_{X'}}$. 

Let $X$ be a complete normal irreducible variety. 
A $\mathbb Q$-b-divisor $\mathbf D$ of $X$ is 
{\em{b-numerically trivial}} (resp.~{\em{$\mathbb Q$-b-linearly trivial}}) 
if there exists a proper birational morphism 
$X'\to X$ from a complete normal variety $X'$ such that 
$\mathbf D=\overline{\mathbf D_{X'}}$ with $\mathbf D_{X'}\equiv 0$ 
(resp.~$\mathbf D_{X'}\sim _{\mathbb Q}0$). 
\end{defn}

For the details of (b-)potentially nef divisors, 
we recommend the reader to see \cite[Section 2]
{fujino-slc-trivial}. 

\section{Quick review of basic slc-trivial fibrations}\label{C-sec3}

In this section, we quickly recall some definitions 
of {\em{basic slc-trivial fibrations}} in \cite[Section 4]{fujino-slc-trivial}. 
We recommend the reader to see \cite[1.15]{fujino-slc-trivial} 
for some historical comments. 

\medskip 

We introduce the notion of basic slc-trivial fibrations. 

\begin{defn}[{Basic slc-trivial fibrations, 
see \cite[Definition 4.1]{fujino-slc-trivial}}]\label{C-def3.1}
A {\em{pre-basic slc-trivial fibration}} $f \colon (X, B)\to Y$ consists of 
a projective surjective morphism 
$f \colon X\to Y$ and a simple normal crossing pair $(X, B)$ satisfying 
the following properties: 
\begin{itemize}
\item[(1)] $Y$ is a normal irreducible variety,   
\item[(2)] every stratum of $X$ is dominant onto $Y$ and 
$f_*\mathcal O_X\simeq \mathcal O_Y$, 
\item[(3)] $B$ is a $\mathbb Q$-divisor such that $B=B^{\leq 1}$ holds 
over 
the generic point of $Y$, and 
\item[(4)] there exists 
a $\mathbb Q$-Cartier $\mathbb Q$-divisor $D$ on $Y$ such that 
$$
K_X+B\sim _{\mathbb Q}f^*D. 
$$ 
\end{itemize}
If a pre-basic slc-trivial fibration $f \colon (X, B)\to Y$ also satisfies 
\begin{itemize}
\item[(5)] $\rank f_*\mathcal O_X(\lceil -B^{<1}\rceil)=1$, 
\end{itemize}
then it is called a {\em{basic slc-trivial fibration}}. 
\end{defn}

Roughly speaking, if $X$ is irreducible 
and $(X, B)$ is sub kawamata log terminal 
(resp.~sub log canonical) over the generic point of $Y$, 
then it is a klt-trivial fibration (resp.~an lc-trivial fibration). 

\medskip 

In order to define discriminant $\mathbb Q$-b-divisors and 
moduli $\mathbb Q$-b-divisors for basic slc-trivial fibrations, 
we need the notion of induced (pre-)basic slc-trivial fibrations. 

\begin{say}[{Induced (pre-)basic slc-tirival 
fibrations, see \cite[4.3]{fujino-slc-trivial}}]\label{C-say3.2}
Let $f \colon (X, B)\to Y$ be a \linebreak 
(pre-)basic slc-trivial fibration 
and let $\sigma \colon Y'\to Y$ be a generically finite 
surjective morphism from a normal irreducible variety $Y'$. 
Then we have an {\em{induced {\em{(}}pre-{\em{)}}basic slc-trivial fibration}} 
$f' \colon (X', B_{X'})\to Y'$, where 
$B_{X'}$ is defined by $\mu^*(K_X+B)=K_{X'}+B_{X'}$, with 
the following commutative diagram: 
$$
\xymatrix{
   (X', B_{X'}) \ar[r]^{\mu} \ar[d]_{f'} & (X, B)\ar[d]^{f} \\
   Y' \ar[r]_{\sigma} & Y, 
} 
$$
where $X'$ coincides with 
$X\times _{Y}Y'$ over a nonempty Zariski open set of $Y'$. 
More precisely, $X'$ is a simple normal crossing variety with a morphism 
$X'\to X\times _Y Y'$ that is an isomorphism over 
a nonempty Zariski open set of $Y'$ such that 
$X'$ is projective over $Y'$ and that every stratum of $X'$ is dominant onto 
$Y'$. 
\end{say}

Now we are ready to define {\em{discriminant 
$\mathbb Q$-b-divisors}} and 
{\em{moduli $\mathbb Q$-b-divisors}} for basic slc-trivial fibrations. 

\begin{say}[{Discriminant and 
moduli $\mathbb Q$-b-divisors, 
see \cite[4.5]{fujino-slc-trivial}}]\label{C-say3.3} 
Let $f \colon (X, B)\to Y$ be a \linebreak 
(pre-)basic 
slc-trivial fibration as in Definition \ref{C-def3.1}. 
Let $P$ be a prime divisor on $Y$. 
By shrinking $Y$ around the generic point of $P$, 
we assume that $P$ is Cartier. We set 
$$
b_P=\max \left\{t \in \mathbb Q\, \left|\, 
\begin{array}{l}  {\text{$(X^\nu, \Theta+t\nu^*f^*P)$ is sub log canonical}}\\
{\text{over the generic point of $P$}} 
\end{array}\right. \right\},  
$$ 
where $\nu \colon X^\nu\to X$ is the normalization and 
$K_{X^\nu}+\Theta=\nu^*(K_X+B)$, that is, 
$\Theta$ is the sum of the inverse images of $B$ and the singular 
locus of $X$, and 
set $$
B_Y=\sum _P (1-b_P)P, 
$$ 
where $P$ runs over prime divisors on $Y$. 
Then it is easy to  see that 
$B_Y$ is a well-defined $\mathbb Q$-divisor on 
$Y$ and is called the {\em{discriminant 
$\mathbb Q$-divisor}} of $f \colon (X, B)\to Y$. We set 
$$
M_Y=D-K_Y-B_Y
$$ 
and call $M_Y$ the {\em{moduli $\mathbb Q$-divisor}} of $f \colon 
(X, B)\to Y$. 
By definition, we have 
$$
K_X+B\sim _{\mathbb Q}f^*(K_Y+B_Y+M_Y). 
$$

Let $\sigma\colon Y'\to Y$ be a proper birational morphism 
from a normal variety $Y'$ and let $f' \colon (X', B_{X'})\to Y'$ be 
an induced (pre-)basic slc-trivial fibration 
by $\sigma \colon Y'\to Y$.  
We can define $B_{Y'}$, $K_{Y'}$ and $M_{Y'}$ such that 
$\sigma^*D=K_{Y'}+B_{Y'}+M_{Y'}$, 
$\sigma_*B_{Y'}=B_Y$, $\sigma _*K_{Y'}=K_Y$ 
and $\sigma_*M_{Y'}=M_Y$. We note that 
$B_{Y'}$ is independent of the choice of $(X', B_{X'})$, 
that is, $B_{Y'}$ is well defined. Hence 
there exist a unique $\mathbb Q$-b-divisor $\mathbf B$ 
such that 
$\mathbf B_{Y'}=B_{Y'}$ for every $\sigma \colon Y'\to Y$ and a unique 
$\mathbb Q$-b-divisor $\mathbf M$ such that $\mathbf M_{Y'}=M_{Y'}$ for 
every $\sigma \colon Y'\to Y$. 
Note that $\mathbf B$ is called the {\em{discriminant $\mathbb Q$-b-divisor}} and 
that $\mathbf M$ is called 
the {\em{moduli $\mathbb Q$-b-divisor}} associated to $f \colon (X, B)\to Y$. 
We sometimes simply say that $\mathbf M$ is 
the {\em{moduli part}} of $f \colon (X, B)\to Y$. 
\end{say}

For the full details of this section, we recommend the reader to see 
\cite[Section 4]{fujino-slc-trivial}. 

\section{On variation of mixed Hodge structure}\label{d-sec4}

This section heavily depends on \cite[Sections 4 and 7]{fujino-fujisawa}. 
We strongly recommend the reader to take a quick look at \cite[Section 4]{fujino-fujisawa} 
before reading this section. 

\medskip 

Let us quickly recall \cite[Theorem 7.1]{fujino-fujisawa}, which is 
one of the main ingredients of 
\cite{fujino-slc-trivial} (see \cite[Section 3]{fujino-slc-trivial}). 

\begin{thm}[{\cite[Theorem 7.1]{fujino-fujisawa}}]\label{d-thm4.1}
Let $(V, T)$ be a simple normal crossing pair such that $T$ is 
reduced and let 
$h \colon V\to Y$ be a projective surjective morphism onto a smooth 
variety $Y$. 
Assume that every stratum of $(V, T)$ is dominant onto $Y$. 
Let $\Sigma$ be a simple normal crossing divisor on $Y$ such that 
every stratum of $(V, T)$ is smooth over $Y^*=Y\setminus \Sigma$. 
We put $V^*=h^{-1}(Y^*)$, $T^*=T|_{V^*}$, and $d=\dim V-\dim Y$. 
Let $\iota \colon  V^*\setminus T^*\hookrightarrow V^*$ be the natural 
open immersion. 
Then the local system $R^k(h|_{V^*})_*\iota_!\mathbb Q_{V^*\setminus T^*}$ 
underlies a graded polarizable admissible variation of 
$\mathbb Q$-mixed Hodge structure on $Y^*$ for every $k$. 
We put $\mathcal V^k_{Y^*}=R^k(h|_{V^*})_*\iota_!\mathbb Q_{V^*\setminus T^*}
\otimes \mathcal O_{Y^*}$ for 
every $k$. 
Let 
$$
\cdots \subset F^{p+1}(\mathcal V^k_{Y^*})\subset F^p(\mathcal V^k_{Y^*})
\subset F^{p-1}(\mathcal V^k_{Y^*})\subset \cdots 
$$
be the Hodge filtration. 
We assume that all the local monodromies on the local system 
$R^k(h|_{V^*})_*\iota_!\mathbb Q_{V^*\setminus T^*}$ around $\Sigma$ 
are unipotent for every $k$. 
Then $R^kh_*\mathcal O_V(-T)$ is isomorphic to the canonical extension 
of 
$$
\Gr ^0_F(\mathcal V^k_{Y^*})=F^0(\mathcal V^k_{Y^*})/F^1(\mathcal V^k_{Y^*}), 
$$ 
which is denoted by $\Gr^0_F(\mathcal V^k_Y)$, for every $k$. 
By taking the dual, we have $$R^{d-k}h_*\omega_{V/Y}(T)\simeq 
\left(\Gr^0_F(\mathcal V^k_Y)\right)^*$$ for every $k$. 
\end{thm}

For the details of Theorem \ref{d-thm4.1}, we recommend the reader to 
see \cite[Sections 4 and 7]{fujino-fujisawa} (see also \cite{ffs}). 
We note that the reader can find basic definitions of variations of 
mixed Hodge structure in \cite[Section 3]{fujino-fujisawa}.  

\medskip 

Let us introduce the notion of {\em{birational maps}} 
of simple normal crossing 
pairs. 

\begin{defn}[Birational maps of simple normal crossing 
pairs]\label{d-def4.2}
Let $(V_1, T_1)$ and $(V_2, T_2)$ be 
simple normal crossing pairs such that 
$T_1$ and $T_2$ are reduced. 
Let $\alpha \colon V_1\dashrightarrow V_2$ be a proper birational map. 
Assume that 
there exist Zariski open sets $U_1$ 
and $U_2$ of $V_1$ and $V_2$ respectively such that 
$U_1$ contains 
the generic point of 
any stratum of $(V_1, T_1)$, $U_2$ contains 
the generic point of any stratum of $(V_2, T_2)$, 
and $\alpha$ induces an isomorphism between 
$(U_1, T_1|_{U_1})$ and $(U_2, T_2|_{U_2})$. 
Then we call $\alpha$ a {\em{birational map between 
$(V_1, T_1)$ and $(V_2, T_2)$}}. 
\end{defn}

As an easy application of \cite[Lemma 6.2]{fujino-fujisawa} and 
\cite[Theorem 1.4]{bierstone}, we can prove the 
following useful lemma. 

\begin{lem}\label{d-lem4.3}
Let $(V_1, T_1)$ and $(V_2, T_2)$ be simple normal crossing 
pairs 
such that $T_1$ and $T_2$ are reduced. 
Let $\alpha \colon V_1\dashrightarrow V_2$ be a birational map 
between $(V_1, T_1)$ and $(V_2, T_2)$. 
Then there exists a commutative diagram 
\begin{equation}\label{d-eq4.1}
\xymatrix{
& (V', T') \ar[dl]_-{p_1}\ar[dr]^-{p_2}&  \\
(V_1, T_1) \ar@{-->}[rr]_-\alpha&& (V_2, T_2),  
}
\end{equation}
where $(V', T')$ is a simple normal crossing pair 
such that $T'$ is reduced, and $p_i$ is a proper 
birational morphism between $(V', T')$ and 
$(V_i, T_i)$ for $i=1, 2$. 
In this situation, $p_i$ induces a natural one-to-one correspondence 
between the set of strata of $(V', T')$ and that of $(V_i, T_i)$ for 
$i=1, 2$. 
Let $S$ be any stratum of $(V', T')$. 
Then we have $$Rp_{i*}\mathcal O_S\simeq 
\mathcal O_{p_i(S)}$$ for 
$i=1, 2$. 
Moreover, we have $$Rp_{i*}\mathcal O_{V'}(-T')\simeq 
\mathcal O_{V_i}(-T_i)$$ for 
$i=1, 2$. 
\end{lem}
\begin{proof}
By \cite[Theorem 1.4]{bierstone}, we can take 
a desired commutative diagram \eqref{d-eq4.1}, where 
$p_i$ is a proper birational morphism 
between $(V', T')$ and $(V_i, T_i)$ such that 
$p_i$ is an isomorphism over $U_i$ for $i=1, 2$. 
By \cite[Lemma 6.2]{fujino-fujisawa}, 
we have $Rp_{i*}\mathcal O_{V'}(-T')\simeq 
\mathcal O_{V_i}(-T_i)$ for $i=1, 2$. 
Let $S$ be a stratum of $(V', T')$. 
Then $p_i(S)$ is a stratum of $(V_i, T_i)$ since 
$p_i$ is a birational 
morphism between $(V', T')$ and $(V_i, T_i)$ for $i=1, 2$. 
Therefore, $p_i(S)$ is a smooth 
irreducible variety and $p_i \colon S\to p_i(S)$ 
is obviously birational for $i=1, 2$. 
This implies that $Rp_{i*}\mathcal O_S\simeq 
\mathcal O_{p_i(S)}$ for $i=1, 2$. 
Since $p_i \colon V'\to V_i$ is a proper 
birational morphism between $(V', T')$ and 
$(V_i, T_i)$, it is easy to see that there 
exists a natural one-to-one correspondence between 
the set of strata of $(V', T')$ and that of 
$(V_i, T_i)$ for $i=1, 2$. 
\end{proof}

\begin{rem}\label{x-rem4.4}
In Lemma \ref{d-lem4.3}, we assume that 
$\alpha\colon (V_1, T_1)\dashrightarrow (V_2, T_2)$ is 
projective over a fixed scheme $Y$, that is, 
there exists the following commutative diagram
$$
\xymatrix{
(V_1, T_1) \ar[dr]_-{h_1}\ar@{-->}[rr]^-\alpha&& (V_2, T_2)
\ar[dl]^-{h_2}\\
&Y& 
}
$$ 
such that $h_1$ and $h_2$ are projective. 
Then we see that we can make $V'$ projective over $Y$ by the 
proof of Lemma \ref{d-lem4.3}. 
\end{rem}

We define a somewhat artificial 
condition for birational maps of simple normal crossing pairs. 
We will use it in Lemma \ref{x-lem4.6} below. 
For the basic definitions of semi-simplicial 
varieties, see, for example, \cite[Section 5.1]{peters-steenbrink}. 

\begin{defn}\label{x-def4.5}
Let $(V, T)$ be a simple normal crossing 
pair such that 
$T$ is reduced. 
Let $\alpha \colon V\dashrightarrow V$ be a birational map 
between $(V, T)$ and $(V, T)$ 
in the sense of Definition \ref{d-def4.2}. 
We say that $\alpha$ satisfies condition $(\bigstar)$ 
if there exists a commutative diagram 
\begin{equation}\label{d-eq4.2}
\xymatrix{
& (V', T') \ar[dl]_-{p_1}\ar[dr]^-{p_2}&  \\
(V, T) \ar@{-->}[rr]_-\alpha&& (V, T)
}
\end{equation}
with the following properties: 
\begin{itemize}
\item[(1)] $(V', T')$ is a simple normal crossing pair such that $T'$ is reduced.  
\item[(2)] $p_i$ is a proper 
birational morphism between $(V', T')$ and $(V, T)$ in the sense of 
Definition \ref{d-def4.2} for $i=1, 2$. 
\item[(3)] There are semi-simplicial resolutions 
$\varepsilon _T \colon  T_\bullet \to T$ 
and $\varepsilon _V \colon  V_\bullet \to V$, that is, 
$T_\bullet$ and $V_\bullet$ are semi-simplicial varieties, $\varepsilon_T$ and 
$\varepsilon _V$ are argumentations and of cohomological descent, 
such that $V_p$ and $T_q$ are disjoint unions of some strata of $(V, T)$ for all 
$p$ and $q$ and 
that they fit in the following commutative diagram 
\begin{equation}\label{d-eq4.3}
\xymatrix{T_\bullet\ar[d]_-{\varepsilon_T}\ar[r]^-\phi& V_\bullet
\ar[d]^-{\varepsilon _V}\\ 
T \ar[r]_j& V, 
}
\end{equation} 
where $\phi$ is a morphism of semi-simplicial varieties 
and $j$ is the natural closed embedding. 
Moreover, $\varepsilon_T \colon  S\to \varepsilon _T(S)$ 
(resp.~$\varepsilon_V \colon S\to \varepsilon _V(S)$) is 
a natural isomorphism for any irreducible component $S$ of $T_\bullet$ 
(resp.~$V_\bullet$). We note that 
$S$ is a stratum of $(V, T)$. 
\item[(4)] There are semi-simplicial varieties $\varepsilon_{T'} \colon 
T'_\bullet 
\to T'$ and $\varepsilon _{V'} \colon V'_\bullet \to V'$ such that 
$\varepsilon _{T'}$ and $\varepsilon _{V'}$ are argumentations, 
$V'_p$ and $T'_q$ are disjoint unions of 
some strata of $(V', T')$ for all $p$ and $q$ 
and that they fit in the following commutative diagram 
\begin{equation}\label{d-eq4.4}
\xymatrix{T'_\bullet\ar[d]_-{\varepsilon_{T'}}\ar[r]^-{\phi'}& V'_\bullet
\ar[d]^-{\varepsilon _{V'}}\\ 
T' \ar[r]_{j'}& V', 
}
\end{equation} 
where $\phi'$ is a morphism of semi-simplicial varieties and 
$j'$ is the natural closed embedding. 
As in (3), $\varepsilon_{T'} \colon S'\to \varepsilon _{T'}(S')$ 
(resp.~$\varepsilon_{V'} \colon S'\to \varepsilon _{V'}(S')$) is 
a natural isomorphism for any irreducible component $S'$ of $T'_\bullet$ 
(resp.~$V'_\bullet$). We note that $S'$ is a stratum of $(V', T')$. 
\item[(5)] The following commutative diagram 
\begin{equation}\label{d-eq4.5}
\xymatrix{&& T' \ar[dll]_-{p_1|_{T'}}\ar[dd]^(.60){j'}\ar[drr]^-{p_2|_{T'}}&& \\ 
T\ar[dd]_-j \ar@{-->}[rrrr]|-\hole ^(.40){\alpha|_T}&&&& T\ar[dd]^-j \\
&& V'\ar[dll]_-{p_1}\ar[drr]^-{p_2} && \\ 
V \ar@{-->}[rrrr]_-\alpha&&&& V
}
\end{equation}
can be lifted to a commutative diagram 
\begin{equation}\label{d-eq4.6}
\xymatrix{&& T'_\bullet \ar[dll]_-{p_1|_{T'_\bullet}}\ar[dd]^(.60){\phi'}
\ar[drr]^-{p_2|_{T'_\bullet}}&& \\ 
T_\bullet\ar[dd]_-\phi \ar@{-->}[rrrr]|-\hole ^(.40){\alpha|_{T_\bullet}}&&&& T_\bullet\ar[dd]^-\phi \\
&& V'_\bullet\ar[dll]_-{p_1|_{V'_\bullet}}\ar[drr]^-{p_2|_{V'_\bullet}} && \\ 
V _\bullet\ar@{-->}[rrrr]_-{\alpha_\bullet}&&&& V_\bullet
}
\end{equation} 
over \eqref{d-eq4.5} 
by \eqref{d-eq4.3} and \eqref{d-eq4.4} 
such that $p_1|_{V'_p}, p_2|_{V'_p}, \alpha_p, p_1|_{T'_q}, 
p_2|_{T'_q}$, and $\alpha|_{T'_q}$ are birational maps 
of smooth varieties for all $p$ and $q$. 
\item[(6)] If $\alpha\colon (V, T)\dashrightarrow (V, T)$ is 
projective over a fixed scheme $Y$, that is, 
there exists the following commutative diagram 
$$
\xymatrix{
(V, T) \ar[dr]_-h\ar@{-->}[rr]^-\alpha&& (V, T)
\ar[dl]^-h\\
&Y& 
}
$$ such that $h$ is projective, 
then 
$V'$ is also projective over $Y$. 
\end{itemize}
\end{defn}

The main purpose of this section is to establish the 
following result, which will play a crucial role in the proof of 
Theorem \ref{B-thm1.3} in Section \ref{C-sec5}. 

\begin{lem}\label{x-lem4.6}
We use the same notation and assumption as in Theorem \ref{d-thm4.1}. 
We assume that $Y$ is a curve. 
We further assume that $(V, T+\Supp h^*\Sigma)$ is a simple 
normal crossing pair and that all the local monodromies 
on the local system $R^jh_*\mathbb Q_{S^*}$ around 
$\Sigma$ are unipotent for any stratum $S$ of $(V, T)$ and 
all $j$, where $S^*=S|_{V^*}$. 
Let $\alpha \colon V\dashrightarrow V$ be a birational map between $(V, T)$ 
and $(V, T)$ over $Y$.  
We assume that 
$\alpha$ satisfies condition $(\bigstar)$ in Definition \ref{x-def4.5}. 
Then $\alpha$ induces isomorphisms  
$$
\alpha^* \colon W_m\Gr^0_F(\mathcal V^k_Y)\overset{\sim}{\longrightarrow}
W_m\Gr^0_F(\mathcal V^k_Y)
$$ 
for all $m$ and $k$, where $W$ denotes the canonical extension of 
the weight filtration. 

Let $G$ be a finite group which acts on $(V, T)$ birationally over $Y$ 
such that every element $\alpha\in G$ satisfies condition $(\bigstar)$ 
in Definition \ref{x-def4.5}. 
Then $G$ acts on $W_m\Gr^0_F(\mathcal V^k_Y)$ for all $m$ and $k$. 
\end{lem}
In the proof of Lemma \ref{x-lem4.6}, we will use some arguments and 
constructions in \cite[Section 4]{fujino-fujisawa}. 
\begin{proof}[Proof of Lemma \ref{x-lem4.6}]
\setcounter{step}{0}
By assumption, $\alpha$ satisfies condition $(\bigstar)$ in Definition \ref{x-def4.5}. 
Therefore, we can take a commutative diagram 
\begin{equation*}\label{eq-zu7}
\xymatrix{
& (V', T') \ar[dl]_-{p_1}\ar[dr]^-{p_2}&  \\
(V, T) \ar@{-->}[rr]_-\alpha&& (V, T)
}
\end{equation*}
as in \eqref{d-eq4.2}. We note that 
$V'$ is projective 
over $Y$. 
From now on, we will use the same notation as in 
Definition \ref{x-def4.5}. 
We put $u=h\circ j\circ\varepsilon_T \colon T_\bullet \to Y$ and 
$v=h\circ \varepsilon _V \colon V_\bullet \to Y$. 
We set $E_\bullet =v^{-1}(\Sigma)_{\mathrm{red}}$ and 
$F_\bullet=u^{-1}(\Sigma)_{\mathrm{red}}$. 
Since $(V, T+\Supp h^*\Sigma)$ is a simple normal crossing 
pair by assumption, 
$E_\bullet$ and $F_\bullet$ are simple 
normal crossing divisors on $V_\bullet$ and $T_\bullet$, respectively. 
As in the proof of \cite[Lemma 4.12]{fujino-fujisawa}, 
we can construct a complex $C(\phi^*)$ 
on $Y$ equipped with filtrations $W$ and $F$ such that 
$H^k(C(\phi^*))\simeq 
\mathcal V^k_Y$, where 
$\mathcal V^k_Y$ is the canonical extension of 
$\mathcal V^k_{Y^*}=R^k(h|_{V^*})_*
\iota_!\mathbb Q_{V^*
\setminus T^*}\otimes \mathcal O_{Y^*}$, for every $k$.   
We note that the filtration $W$ is denoted by $L$ in \cite[Lemma 4.12]
{fujino-fujisawa}. 

\begin{step}\label{d-step1}
The spectral sequence 
$$
E^{p,q}_1(C(\phi^*), F)=H^{p+q}(\Gr^p_FC(\phi^*))\Rightarrow 
H^{p+q}(C(\phi^*))
$$ 
degenerates at $E_1$ 
(see the proof of \cite[Lemma 4.12]{fujino-fujisawa} 
and \cite[13.3]{fujino-slc-trivial}). 
Therefore, we have the following short exact sequences 
$$
\xymatrix{
0 \ar[r]& \ar[r]^-{s^{p+q}}H^{p+q}(F^1C(\phi^*)) & \ar[r]^-{t^{p+q}}H^{p+q}(C(\phi^*)) 
& H^{p+q}(\Gr^0_FC(\phi^*)) \ar[r]& 0
}
$$
for all $p$ and $q$. 
We note that $F^0C(\phi^*)=C(\phi^*)$ by construction. 
Let us consider the following commutative diagram. 
$$
\xymatrix{
0 \ar[r]& \ar[r]^-{s^{p+q}}H^{p+q}(F^1C(\phi^*)) & \ar[r]^-{t^{p+q}}H^{p+q}(C(\phi^*)) 
& H^{p+q}(\Gr^0_FC(\phi^*)) \ar[r]& 0\\ 
& & H^{p+q} (W_{-p}C(\phi^*))\ar[r]\ar[u]_-{a^{p+q}_{-p}}& 
H^{p+q}(W_{-p} \Gr^0_FC(\phi^*))\ar[u]_-{b^{p+q}_{-p}}& 
}
$$ 
By definition, 
we have 
$$
F^1H^{p+q}(C(\phi^*))=\mathrm{Im}\, s^{p+q}
$$ 
and 
$$
W_q H^{p+q}(C(\phi^*))=\mathrm{Im}\,a^{p+q}_{-p}
$$ 
for all $p$ and $q$. 
We put 
\begin{equation}\label{d-eq4.7}
W_qH^{p+q}(\Gr^0_FC(\phi^*)):=\mathrm{Im}\, b^{p+q}_{-p}
\end{equation}
for all $p$ and $q$. 
Then the map $t^{p+q}$ 
induces 
\begin{equation}\label{d-eq4.8}
\xymatrix{
\Gr^0_FH^{p+q}(C(\phi^*)) \ar[r]^-{\sim} & H^{p+q} (\Gr^0_FC(\phi^*)) \\ 
W_q \Gr^0_FH^{p+q} (C(\phi^*)) \ar[r]_-{i^{p+q}_q}\ar@{^{(}->}[u]& 
W_q H^{p+q} (\Gr^0_FC(\phi^*)) \ar@{^{(}->}[u]
}
\end{equation}
for all $p$ and $q$. 
We will prove that $i^{p+q}_q$ are isomorphisms 
for all $p$ and $q$ in Step \ref{d-step2}. 
\end{step}
\begin{step}\label{d-step2}
Let us analyse the spectral sequence 
\begin{equation}\label{d-eq4.9} 
E^{p, q}_1(C(\phi^*), W)\Rightarrow H^{p+q}(C(\phi^*))
\end{equation} 
in detail. 
Let $\Omega_{V_{p+1}/Y}(\log E_{p+1})$ and 
$\Omega_{T_p/Y}(\log F_p)$ be relative logarithmic de Rham complexes 
of $v_{p+1} \colon V_{p+1}\to Y$ and $u_p \colon T_p\to Y$, respectively. 
Then we have  
$$
\left(E^{p, q}_1(C(\phi^*), W), F\right)=
\left(R^q(v_{p+1})_*\Omega_{V_{p+1}/Y}(\log E_{p+1}), F\right)
\oplus \left(R^q(u_p)_*\Omega_{T_p/Y}(\log F_p), F\right) 
$$ 
by construction. We note that the 
differentials of the spectral sequence \eqref{d-eq4.9}  
are strictly compatible with the filtration induced by $F$ 
(see \cite[(1.1.5)]{deligne}, \cite[Remark 3.2]{fujino-fujisawa}, 
and \cite[A.~3.1]{peters-steenbrink}) 
and that the spectral sequence \eqref{d-eq4.9} 
degenerates at $E_2$. 
We do not repeat the proof of the above facts here. 
For the proof, see the first part of the proof of 
\cite[Lemma 4.12]{fujino-fujisawa} 
and \cite[13.3]{fujino-slc-trivial}.  

The following argument corresponds to the strictness of the filtration $F$ 
on the $E_0$-term of the spectral sequence 
$E^{p,q}_r(C(\phi^*), W)$ (see \cite[13.3]{fujino-slc-trivial}). 
By \cite[(2.11) Theorem]{steenbrink}, 
$R^b(u_p)_*\Omega^a_{T_p/Y}(\log F_p)$ is 
locally free for any $a$, $b$, and $p$. 
Therefore, the spectral sequence 
$$
R^b(u_p)_*\Omega^a_{T_p/Y}(\log F_p)
\Rightarrow 
R^{a+b}(u_p)_*\Omega_{T_p/Y}(\log F_p)
$$ 
degenerates at $E_1$. In particular, 
$$
\Gr^0_FR^q(u_p)_*\Omega_{T_p/Y}(\log F_p)\simeq 
R^q(u_p)_*\mathcal O_{T_p}
$$ 
holds for any $p$, $q$. 
By the same way, we see 
that 
$$\Gr^0_FR^q(v_{p+1})_*\Omega_{V_{p+1}/Y}(\log E_{p+1})\simeq 
R^q(v_{p+1})_*\mathcal O_{V_{p+1}}
$$ 
holds for any $p$, $q$. 
Thus we have 
\begin{equation}\label{d-eq4.10}
\begin{split}
&\Gr^0_FE^{p,q}_1(C(\phi^*), W) \\&=\Gr^0_FR^q(v_{p+1})_*
\Omega_{V_{p+1}/Y}(\log E_{p+1}) 
\oplus \Gr^0_F R^q(u_p)_*\Omega_{T_p/Y}(\log F_p)\\ 
&\simeq R^q(v_{p+1})_*\mathcal O_{V_{p+1}}\oplus 
R^q(u_p)_*\mathcal O_{T_p}. 
\end{split} 
\end{equation}

By taking $\Gr^0_F$ of the spectral sequence 
\eqref{d-eq4.9}, 
we obtain the following spectral sequence 
\begin{equation*}
E^{p,q}_1(\Gr^0_FC(\phi^*), W)\Rightarrow 
H^{p+q}(\Gr^0_FC(\phi^*)). 
\end{equation*} 
Note that 
$$
\Gr^0_FE^{p,q}_1(C(\phi^*), W)\simeq 
E^{p,q}_1(\Gr^0_FC(\phi^*), W)
$$ 
holds as we saw in \eqref{d-eq4.10}. Moreover, 
$$
\Gr^0_FE^{p, q}_r(C(\phi^*), W) \simeq E^{p, q}_r(\Gr^0_FC(\phi^*), W)
$$ 
holds 
for every $r\geq 0$ 
by the lemma on two filtrations 
(see \cite[Propositions (7.2.5) and (7.2.8)]{deligne2} and 
\cite[Theorem 3.12]{peters-steenbrink}). 
Hence, we obtain 
\begin{equation}\label{d-eq4.11}
\begin{split}
\Gr^0_F\Gr^W_q\!H^{p+q}(C(\phi^*))&\simeq \Gr^0_FE^{p,q}_2(C(\phi^*), W)
\\ &\simeq E^{p,q}_2(\Gr^0_FC(\phi^*), W) 
\simeq 
\Gr^W_q\!H^{p+q}(\Gr^0_FC(\phi^*))
\end{split}
\end{equation}
for all $p$ and $q$. 
We note that the filtration $W$ on $H^{p+q}(\Gr^0_FC(\phi^*))$ 
is the one defined in \eqref{d-eq4.7}. We also note that 
$\Gr^0_F\Gr^W_q\!H^{p+q}(C(\phi^*))$ is canonically 
isomorphic to $\Gr^W_q\!\Gr^0_FH^{p+q}(C(\phi^*))$. 
Thus, we can check that 
\begin{equation}\label{d-eq4.12}
\xymatrix{
i^{p+q}_q \colon W_q\Gr^0_FH^{p+q}(C(\phi^*)) \ar[r]&
W_qH^{p+q}(\Gr^0_FC(\phi^*))
}
\end{equation} 
in \eqref{d-eq4.8} are isomorphisms 
for all $p$ and $q$ inductively by using \eqref{d-eq4.8} and \eqref{d-eq4.11} . 
\end{step}

\begin{step}\label{d-step3} 
In this proof, we did not define the filtration $W$ on 
$C(\phi^*)$ explicitly. For the details of 
the filtration $W$ on $C(\phi^*)$, 
which is denoted by 
$L$ in \cite[Section 4]{fujino-fujisawa}, see 
(4.2.1) and (4.8.2) in \cite[Section 4]{fujino-fujisawa}.
By construction, we have  
\begin{equation*}
\begin{split}
W_{-p} \Gr^0_FC(\phi^*)^n &=
W_{-p-1}(Rv_*\mathcal O_{V_\bullet})^{n+1}\oplus 
W_{-p}(Ru_*\mathcal O_{T_\bullet})^n\\ 
&= \bigoplus_{s\geq p+1} (R(v_s)_*\mathcal O_{V_s})^{n+1-s}
\oplus \bigoplus _{t\geq p} (R(u_t)_*\mathcal O_{T_t})^{n-t}. 
\end{split}
\end{equation*} 
Therefore, by Lemma \ref{d-lem4.3} and the commutative 
diagram \eqref{d-eq4.6} in Definition \ref{x-def4.5}, 
$\alpha$ induces isomorphisms 
$$
\alpha^* \colon W_{-p}\Gr^0_FC(\phi^*)\overset{\sim}{\longrightarrow} 
W_{-p}\Gr^0_FC(\phi^*)
$$ 
for all $p$. 
Thus $\alpha$ induces isomorphisms 
$$
\alpha^* \colon W_qH^{p+q}(\Gr^0_FC(\phi^*))
\overset{\sim}{\longrightarrow} W_qH^{p+q}(\Gr^0_FC(\phi^*))
$$ 
for all $p$ and $q$ by the following commutative diagram 
\begin{equation*}
\xymatrix{
H^{p+q}(W_{-p}\Gr^0_FC(\phi^*)) \ar[d]^-\wr_-{\alpha^*}\ar[r]
& H^{p+q}(\Gr^0_FC(\phi^*))\ar[d]^-\wr_-{\alpha^*}
\\
H^{p+q}(W_{-p}\Gr^0_FC(\phi^*))\ar[r]&
H^{p+q}(\Gr^0_FC(\phi^*))
}
\end{equation*} 
and the definition of the filtration $W$ in \eqref{d-eq4.7}. 
Hence, we obtain isomorphisms 
$$
\alpha^* \colon W_mH^k(\Gr^0_FC(\phi^*))\overset{\sim}
{\longrightarrow}W_mH^k(\Gr^0_FC(\phi^*))
$$ 
for all $m$ and $k$ by putting 
$p=k-m$ and $q=m$. 
By \eqref{d-eq4.12} and the fact that 
$\mathcal V^k_Y\simeq H^k(C(\phi^*))$, we obtain the 
desired isomorphisms 
$$
\alpha^* \colon W_m\Gr^0_F(\mathcal V^k_Y)\overset{\sim}
{\longrightarrow}W_m\Gr^0_F(\mathcal V^k_Y)
$$ 
for all $m$ and $k$. 
\end{step}
When the group $G$ acts on $(V, T)$ birationally over $Y$ 
such that every element $\alpha\in G$ satisfies 
condition $(\bigstar)$ in Definition 
\ref{x-def4.5}, it is easy to see that 
$G$ also acts on $W_m\Gr^0_F(\mathcal V^{k}_Y)$ for all $m$ and 
$k$ by the 
above result. 
\end{proof}

We make an important remark on dual variations of 
mixed Hodge structure. We will use it 
in Step \ref{C-step4} in the proof of Theorem \ref{B-thm1.3}.  

\begin{rem}
[{see \cite[Remarks 3.15 and 7.4]{fujino-fujisawa}}]\label{x-rem4.7}
We use the same notation and assumption as in Lemma \ref{x-lem4.6}. 
Let us consider the dual local system of $R^k(h|_{V^*})_*\iota_!\mathbb Q_{
V^*\setminus T^*}$ and 
the dual variation of mixed Hodge structure on it. 
Then the locally free sheaf $(\mathcal V^k_{Y^*})^*$ carries 
the Hodge filtration $F$ and the weight filtration $W$ defined 
as in \cite[Remark 3.15]{fujino-fujisawa}. 
By the construction of the Hodge filtration $F$, 
$$
\Gr^0_F\!\left((\mathcal V^k_Y)^*\right)\simeq 
\left(\Gr^0_F(\mathcal V^k_Y)\right)^*
$$
holds, where $(\mathcal V^k_Y)^*$ is 
the canonical extension of $(\mathcal V^k_{Y^*})^*$. 
More generally, 
$$
\Gr^{-p}_F\!\left((\mathcal V^k_Y)^*\right)\simeq 
\left(\Gr^p_F(\mathcal V^k_Y)\right)^*
$$
holds for every $p$. 
We note that  
$\Gr^0_F\!\left((\mathcal V^k_Y)^*\right)=F^0\!\left((\mathcal V^k_Y)^*\right)$, 
the canonical extension of the lowest piece of the Hodge filtration. 
By taking the dual of Lemma \ref{x-lem4.6}, 
$G$ acts on $W_m\Gr^0_F\!\left((\mathcal V^k_Y)^*\right)$ for 
every $m$, 
where $W$ denotes the canonical extension of 
the weight filtration of $(\mathcal V^k_{Y^*})^*$. 
We note that we have  
$$
\Gr^W_m\Gr^p_F\!\left((\mathcal V^k_Y)^*\right)
\simeq \left( \Gr^W_{-m}\Gr^{-p}_F(\mathcal V^k_Y)\right)^*
$$ 
for all $p$ and $m$ by construction. 
\end{rem}

We close this section with the following lemma, which is 
more or less well known to the experts (see \cite{zucker}, 
\cite{peters}, \cite{kollar}, and \cite{fujino-fujisawa2}). 
We will use it in the proof of Theorem \ref{B-thm1.3} in Section \ref{C-sec5}. 

\begin{lem}\label{d-lem4.7} 
Let $C$ be a smooth projective curve 
and let $C_0$ be a non-empty Zariski open set of $C$. 
Let $V_0$ be a polarizable variation of $\mathbb Q$-Hodge 
structure over $C_0$ with unipotent monodromies around $\Sigma=C\setminus 
C_0$. Let $F^b$ be the canonical extension of the lowest piece of 
the Hodge filtration. 
Let $\mathcal L$ be a line bundle on $C$ which is a direct summand of $F^b$. 
Assume that $\deg_C\mathcal L=0$. 
Then $\mathcal L|_{C_0}$ is a flat subbundle of $F^b|_{C^0}$. 
\end{lem}

\begin{proof}
Let $h_0$ be the smooth hermitian metric on $\mathcal L|_{C_0}$ induced 
by the Hodge metric of $F^b|_{C_0}$. 
Then $\frac{\sqrt{-1}}{2\pi} \Theta_{h_0}(\mathcal L|_{C_0})$ is a semipositive 
smooth $(1, 1)$-form on $C_0$. 
We note that $\Theta_{h_0}(\mathcal L|_{C_0})$ is the curvature tensor of the Chern 
connection of $(\mathcal L|_{C_0}, h_0)$. 
Then $$\deg_C\mathcal L=\frac{\sqrt{-1}}{2\pi}\int _{C_0} 
\Theta_{h_0}(\mathcal L|_{C_0})$$ holds (see, for example, \cite[Theorem 5.1]{kollar}). 
Note that the right hand side is an improper integral.  
By assumption, $\deg _C\mathcal L=0$. 
This implies that $\Theta_{h_0}(\mathcal L|_{C_0})=0$. 
Therefore, $\mathcal L|_{C_0}$ is a flat subbundle of $F^b|_{C_0}$. 
\end{proof}

\begin{rem}\label{d-rem4.8}
In Lemma \ref{d-lem4.7}, 
the smooth hermitian metric $h_0$ on $\mathcal L|_{C_0}$ can be 
extended naturally to a singular hermitian metric $h$ 
on $\mathcal L$ in the sense of Demailly such that 
$\sqrt{-1}\Theta_h(\mathcal L)$ is positive in the sense of 
currents and that the Lelong 
numbers of $h$ are zero everywhere. 
For the details, see \cite[Theorem 1.1]{fujino-fujisawa2}. 
\end{rem}

\section{Proof of Theorem \ref{B-thm1.3}}\label{C-sec5}

In this section, we prove Theorem \ref{B-thm1.3} and 
Corollary \ref{B-cor1.4}.  

\medskip 

Let us prepare an easy lemma. 
By this lemma, we can reduce the problem to the case where the base 
space is a curve. 

\begin{lem}\label{C-lem5.1}
Let $Y$ be a smooth projective irreducible variety with $\dim Y\geq 2$ and 
let $N$ be a numerically trivial Cartier divisor on $Y$. 
Let $H$ be a smooth ample Cartier divisor on $Y$ such 
that $H$ contains no irreducible components of $\Supp N$. 
Then $N\sim 0$ if and only if $N|_H\sim 0$. 
\end{lem}
\begin{proof}
We consider the following long exact sequence 
\begin{equation*}
\begin{split}
0&\to H^0(Y, \mathcal O_X(N-H))\to 
H^0(Y, \mathcal O_Y(N))\to H^0(H, \mathcal O_H(N|_H))
\\&\to H^1(Y, \mathcal O_Y(N-H))\to \cdots. 
\end{split}
\end{equation*} 
It is obvious that $H^0(Y, \mathcal O_Y(N-H))=0$. 
By the Kodaira vanishing theorem, 
we have $H^1(Y, \mathcal O_Y(N-H))=0$. 
Therefore, 
$H^0(Y, \mathcal O_Y(N))\simeq 
H^0(H, \mathcal O_H(N|_H))$ holds. 
In particular, $N\sim 0$ if and only if $N|_H\sim 0$. 
\end{proof}

Let us start the proof of Theorem \ref{B-thm1.3}. 
We adapt Floris's proof of 
Theorem \ref{B-thm1.3} for lc-trivial 
fibrations (see \cite{floris}) to our setting, that is, 
basic slc-trivial fibrations. 

\begin{proof}[Proof of Theorem \ref{B-thm1.3}]
This proof heavily depends on \cite[Section 6]{fujino-slc-trivial}. 
Let $\sigma\colon Y'\to Y$ be a projective 
birational morphism from a smooth 
projective variety $Y'$. By considering 
the induced basic slc-trivial fibration by $\sigma\colon Y'\to Y$, 
we may assume that $Y$ is a smooth projective variety. 
\setcounter{step}{0}
\begin{step}\label{C-step1} 
In this step, we construct a cyclic cover of the generic 
fiber of $f:X\to Y$ following \cite[6.1 and 6.2]{fujino-slc-trivial}. 
Let $f \colon (X, B)\to Y$ be a basic slc-trivial fibration. 
Let $F$ be a general fiber of $f \colon X\to Y$. 
We put 
$$
b=\min\{m \in \mathbb Z_{>0}\, |\, m(K_F+B_F)=m(K_X+B)|_F\sim 0\}.  
$$ 
Then we can write 
\begin{equation}\label{C-eq5.1} 
K_X+B+\frac{1}{b}(\varphi)=f^*(K_Y+B_Y+M_Y)
\end{equation}
with $\varphi\in \Gamma(X, \mathcal K^*_X)$, 
where $B_Y$ is the discriminant $\mathbb Q$-divisor 
and $M_Y$ is the moduli $\mathbb Q$-divisor 
of $f \colon (X, B)\to Y$. 
By taking some suitable blow-ups 
(see \cite[Theorem 1.4 and Section 8]{bierstone} and 
\cite[Lemma 2.11]{fujino-ann}), 
we may assume that $\Supp (B-f^*(B_Y+M_Y))$ is a simple 
normal crossing divisor on $X$, 
$(B^h)^{=1}$ is Cartier, 
and every stratum of $(X, (B^h)^{=1})$ is dominant 
onto $Y$. 
We take the $b$-fold cyclic cover $\pi \colon \widetilde X\to X$ associated 
to \eqref{C-eq5.1}, that is, 
$$
\widetilde X=\Spec _X \bigoplus _{i=0}^{b-1} 
\mathcal O_X(\lfloor i\Delta\rfloor), 
$$
where $\Delta=K_{X/Y}+B-f^*(B_Y+M_Y)$. 
We note that $\pi:\widetilde X\to X$ is a finite Galois 
cover by construction (see \cite[Proposition 6.3 (i)]{fujino-slc-trivial}). 
We put $K_{\widetilde X}+B_{\widetilde X}=\pi^*(K_X+B)$. 
By construction, it is easy to see that 
$(B^h_{\widetilde X})^{=1}=\pi^*((B^h)^{=1})$ and 
that $(\widetilde X, (B^h_{\widetilde X})^{=1})$ is 
semi-log canonical. 
Moreover, every slc stratum of $(\widetilde X, (B^h_{\widetilde X})^{=1})$ 
is dominant onto $Y$. 
We take a projective 
birational morphism $d \colon V\to \widetilde X$
from a simple normal crossing variety $V$ such that 
$d$ is an isomorphism 
over the generic point of every slc stratum of $(\widetilde X, 
(B^h_{\widetilde X})^{=1})$ by \cite[Theorem 1.4]{bierstone}. We 
put $K_V+B_V=d^*(K_{\widetilde X}+B_{\widetilde X})$. 
Then we 
get the following commutative diagram 
\begin{equation}\label{C-eq5.2}
\xymatrix{
(X, B)\ar[d]_-f & \widetilde X \ar[dl]_-{\widetilde f}
\ar[l]_-\pi& (V, B_V)\ar[dll]^-h\ar[l]_-d\\ 
Y & & 
}
\end{equation} 
with $g=\pi\circ d$. 
By taking a suitable birational modification of $Y$ and 
considering induced (pre-)basic slc-trivial fibrations as in 
\cite[6.2]{fujino-slc-trivial}, 
we may further assume that the following properties hold for 
$$
K_X+B+\frac{1}{b}(\varphi)=f^*(K_Y+B_Y+M_Y)
$$ 
and 
$$
h \colon (V, B_V)
\overset{g}{\longrightarrow} (X, B)\overset{f}{\longrightarrow} Y.  
$$
\begin{itemize}
\item[(a)] $Y$ is a smooth projective 
irreducible variety, and $X$ and $V$ are projective 
simple 
normal crossing varieties. 
\item[(b)] There exist simple normal crossing divisors 
$\Sigma_X$, $\Sigma_V$, and $\Sigma_Y$ 
on $X$, $V$, and $Y$, respectively. 
\item[(c)] $f$ and $h$ are projective surjective 
morphisms. 
\item[(d)] The supports of 
$B$, $B_V$, and $B_Y$, $M_Y$ are 
contained in 
$\Sigma_X$, $\Sigma_V$, and $\Sigma_Y$, respectively. 
\item[(e)] Every stratum of $(X, \Sigma^h_X)$ 
and $(V, \Sigma^h_V)$ is smooth 
over $Y\setminus \Sigma_Y$. 
\item[(f)] $f^{-1}(\Sigma_Y)\subset \Sigma_X$, $f(\Sigma^v_X)\subset 
\Sigma_Y$, and $h^{-1}(\Sigma_Y)\subset 
\Sigma_V$, $h(\Sigma^v_V)\subset 
\Sigma_Y$. 
\item[(g)] $(B^h)^{=1}$ and $(B^h_V)^{=1}$ are Cartier. 
\end{itemize}
We note that conditions (a)--(g) above are nothing but the conditions 
stated just before \cite[Proposition 6.3]{fujino-slc-trivial}. 
As we saw in the proof of \cite[Theorem 1.2]{fujino-slc-trivial} 
(see \cite[Section 9]{fujino-slc-trivial}), 
$\mathbf M=\overline{\mathbf M_Y}$ holds and 
$\mathbf M_Y$ is a nef $\mathbb Q$-divisor on $Y$. 
By assumption, $\mathbf M_Y\equiv 0$. 
If $\nu\colon Y''\to Y$ is a finite surjective morphism from a 
smooth 
projective irreducible variety $Y''$, 
then it is easy to see that $\mathbf M_Y\sim _{\mathbb Q}0$ if and 
only if $\nu^*\mathbf M_Y\sim _{\mathbb Q}0$. 
Therefore, by taking a unipotent reduction 
(see \cite[Lemma 7.3]{fujino-slc-trivial}), 
we may further assume that 
\begin{itemize}
\item[(A)] for any irreducible component $P$ of 
$\Supp \Sigma_Y$, there exists a prime divisor $Q$ on $V$ such that 
$\mult _Q(-B_V+h^*B_Y)=0$, 
$h(Q)=P$, and 
$\mult _Qh^*P=1$, 
\item[(B)] all the local monodromies on the local system 
$$R^{\dim V-\dim Y} (h|_{V^*})_*\iota_!\mathbb Q_{V^*\setminus 
(B_{V^*}^h)^{=1}}$$ around 
$\Sigma_Y$ are unipotent, where 
$Y^*=Y\setminus \Sigma_Y$, $V^*=h^{-1}(Y^*)$, $B_{V^*}=(B_V)|_{V^*}$, 
and $\iota\colon V^*\setminus (B_{V^*}^h)^{=1}\hookrightarrow V^*$ is the natural 
open immersion, and 
\item[(C)] all the local monodromies on the local system 
$R^kh_*\mathbb Q_{S^*}$ around $\Sigma_Y$ are unipotent for any stratum $S$ 
of $(V, (B^h_V)^{=1})$ and every $k$, where $S^*=S|_{V^*}$.  
\end{itemize}
Note that the above assumptions (A) and (B) are nothing but 
the assumptions in (iv) and (v) in \cite[Proposition 6.3]{fujino-slc-trivial}. 
We also note that 
we do not treat the assumption (C) in the original statement of 
\cite[Lemma 7.3]{fujino-slc-trivial}. 
Therefore, we have to make $N_j$ in the proof of 
\cite[Lemma 7.3]{fujino-slc-trivial} sufficiently divisible 
in order to make the monodromy on the local system 
$R^kh_*\mathbb Q_{S^*}$ around 
$P_j$, an irreducible component of $\Sigma_Y$, unipotent for any 
stratum $S$ of $(V, (B^h_V)^{=1})$ and 
every $k$ when we take a finite cover $\nu \colon Y''\to Y$ for a unipotent 
reduction (see \cite[Lemma 7.3]{fujino-slc-trivial}). 
\end{step}
\begin{step}\label{C-step2}
We assume that $\dim Y\geq 2$. 
Then we take a general ample Cartier divisor 
$H$ on $Y$ and put $Z=f^*H$ and 
$W=h^*H$. 
In this situation, 
$$
K_X+Z+B+\frac{1}{b}(\varphi)=f^*(K_Y+H+B_Y+M_Y). 
$$ 
By adjunction, 
$$
K_Z+B|_Z+\frac{1}{b}(\varphi|_Z)=f^*(K_H+B_Y|_H+M_Y|_H)
$$ 
holds. 
It is not difficult to see that 
$f|_Z \colon (Z, B|_Z)\to H$ is a basic slc-trivial fibration 
and $$h|_W \colon (W, B_V|_W)\overset {g|_W}
{\longrightarrow} (Z, B|_Z)\overset{f|_Z}{\longrightarrow} 
H$$ satisfies conditions (a)--(g), 
(A), (B), and (C) in Step \ref{C-step1}. 
We note that $B_Y|_H=B_H$ and $M_Y|_H=M_H$ hold, 
where $B_H$ (resp.~$M_H$) is 
the discriminant (resp.~moduli) $\mathbb Q$-divisor 
of $f|_Z \colon (Z, B|_Z)\to H$. 
By Lemma \ref{C-lem5.1}, 
$M_Y\sim _{\mathbb Q}0$ if and only if $M_Y|_H\sim 
_{\mathbb Q}0$. 
Therefore, we can replace $f \colon (X, B)\to Y$ with $f|_Z \colon 
(Z, B|_Z)\to 
H$. 
By repeating this reduction finitely many times, 
we may assume that $Y$ is a smooth projective 
curve. 
\end{step}

\begin{step}\label{C-step3} 
In Step \ref{C-step1}, 
we have already seen that $\pi:\widetilde X\to X$ is Galois. 
Let $G=\mathbb Z/b\mathbb Z$ be the Galois group of 
the $b$-fold cyclic cover 
$\pi \colon \widetilde X\to X$. 
The action of $G$ on $\widetilde X$ preserves the slc strata 
of $(\widetilde X, (B^h_{\widetilde X})^{=1})$ by construction. 
Therefore, 
any element $\alpha$ of $G$ induces a birational map between 
$(V, T)$ and $(V, T)$ over $X$, where 
$T=(B^h_V)^{=1}$. 
From now on, we will check that $\alpha$ satisfies 
condition $(\bigstar)$ in Definition \ref{x-def4.5}. 
As usual, we can take a commutative diagram 
\begin{equation*}
\xymatrix{
& (V', T') \ar[dd]^(.35){g'}\ar[dl]_-{p_1}\ar[dr]^-{p_2}&  \\
(V, T) \ar[dr]_-g\ar[ddr]_-h\ar@{-->}[rr]^(.40){\alpha}&& (V, T)\ar[dl]^-g
\ar[ddl]^-h\\ 
& X\ar[d]^-f& \\ 
& Y& 
}
\end{equation*}
by using \cite[Theorem 1.4]{bierstone}, where $(V', T')$ is a simple 
normal crossing pair such that 
$T'$ is reduced, and $p_i$ is a projective 
birational morphism between $(V', T')$ and $(V, T)$ for 
$i=1, 2$. 
We put $C=(B^h)^{=1}$. 
The irreducible decomposition of $X$ and $C$ are given by 
$$
X=\bigcup _{i\in I} X_i, \quad\text{and}
\quad C=\bigcup _{\lambda\in \Lambda} C_\lambda
$$ 
respectively as in \cite[4.14]{fujino-fujisawa}. 
We put $V=\bigcup _{i\in I}V_i$ and $V_i=\bigcup _j V_{i_j}$, 
where $V_{i_j}$ runs over irreducible components 
of $V$ such that 
$g(V_{i_j})=X_i$. 
We put $T=\bigcup _{\lambda\in \Lambda} T_\lambda$ 
and $T_\lambda=\bigcup _l T_{\lambda_l}$, 
where $T_{\lambda_l}$ runs over irreducible components of $T$ 
such that $g(T_{\lambda_l})=C_\lambda$. 
Note that $T_{\lambda}$ and $V_i$ are disjoint unions 
of some strata of $(V, T)$. 
By applying the same construction as above to 
$(V', T')$ and $g':=g\circ p_1=g\circ p_2 \colon V'\to X$, 
we get 
$V'=\bigcup _{i\in I} V'_i$ and $T'=\bigcup _{\lambda\in \Lambda}
T'_\lambda$. 
We apply the same construction as in \cite[4.14]{fujino-fujisawa} 
to $V=\bigcup _{i\in I}V_i$ and $T=\bigcup _{\lambda\in \Lambda} 
T_\lambda$ 
(resp.~$V'=\bigcup _{i \in I} V'_i$ and $T'=\bigcup _{\lambda\in \Lambda} 
T'_\lambda$) instead of 
$X=\bigcup _{i\in I}X_i$ and $D=\bigcup _{\lambda\in \Lambda} 
D_\lambda$ in \cite[4.14]{fujino-fujisawa}. 
Then we can construct semi-simplicial 
resolutions $\varepsilon _T \colon T_\bullet \to T$ and 
$\varepsilon _V \colon V_\bullet \to V$ 
(resp.~$\varepsilon _{T'} \colon T'_\bullet \to T'$ and $\varepsilon _{V'} 
\colon V'_\bullet \to V'$). 
By construction, these semi-simplicial 
resolutions satisfy the conditions stated in Definition \ref{x-def4.5}. 
Therefore, $\alpha$ satisfies condition $(\bigstar)$. 
This is what we wanted. 
\end{step}
\begin{step}\label{C-step4}
We note that $M_Y$ is a Cartier divisor on $Y$ and 
that $\mathcal O_Y(M_Y)$ is a direct summand 
of 
$$\left(\Gr^0_F(\mathcal V^d_Y)\right)^*\simeq 
\Gr^0_F\!\left((\mathcal V^d_Y)^*\right), 
$$ where 
$d=\dim X-\dim Y$ (see \cite[Proposition 6.3]{fujino-slc-trivial}). 
More precisely, by construction, 
$\mathcal O_Y(M_Y)$ is an eigensheaf of rank one corresponding 
to the eigenvalue 
$\zeta^{-1}$ of 
$$h_*\omega_{V/Y}\left((B^h_V)^{=1}\right)\simeq 
\Gr^0_F\!\left((\mathcal V^d_Y)^*\right)$$ 
by the group action of $G=\mathbb Z/b\mathbb Z$, 
where $\zeta$ is a fixed primitive $b$-th root of unity 
(see the proof of \cite[Proposition 6.3]{fujino-slc-trivial}). 
We take an integer $l$ such that 
\begin{equation*}
\mathcal O_Y(M_Y)\subset W_l\Gr^0_F\!\left((\mathcal V^d_Y)^*\right) 
\ \ 
\text{and} \quad 
\mathcal O_Y(M_Y)\not \subset W_{l-1}\Gr^0_F\!\left((\mathcal V^d_Y)^*\right)
\end{equation*} 
hold.  
Thus we can easily see that 
$\mathcal O_Y(M_Y)$ is an eigensheaf of rank one corresponding 
to the eigenvalue $\zeta^{-1}$ of $W_l\Gr^0_F\!\left((\mathcal V^d_Y)^*\right)$ 
and that $\mathcal O_Y(M_Y)\cap W_{l-1}\Gr^0_F\!\left((\mathcal V^d_Y)^*\right)=\{0\}$ 
in $W_l\Gr^0_F\!\left((\mathcal V^d_Y)^*\right)$.  We note that 
$G$ acts on $W_m\Gr^0_F\!\left((\mathcal V^d_Y)^*\right)$ for every $m$ by 
Lemma \ref{x-lem4.6} and Remark \ref{x-rem4.7}. 
Since $\deg M_Y=0$ by assumption, 
$\mathcal O_Y(M_Y)|_{Y^*}$ defines a local subsystem of 
$\Gr^W_l\!\left((\mathcal V^d_{Y^*})^*\right)$ 
by Lemma \ref{d-lem4.7}. 
We note that 
$$\Gr^W_l\Gr^0_F\!\left((\mathcal V^d_{Y^*})^*\right)\simeq 
\Gr^0_F\Gr^W_l\!\left((\mathcal V^d_{Y^*})^*\right)
=F^0\Gr^W_l\!\left((\mathcal V^d_{Y^*})^*\right)\subset 
\Gr^W_l\!\left((\mathcal V^d_{Y^*})^*\right)$$ 
holds since we have 
$F^1\Gr^W_l\!\left((\mathcal V^d_{Y^*})^*\right)=0$ by 
the construction of the dual Hodge filtration (see \cite[Remark 3.15]
{fujino-fujisawa} and Remark \ref{x-rem4.7}). 
Therefore, there exists a positive integer $a$ 
such that $\mathcal O_Y(aM_Y)|_{Y^*}\simeq 
\mathcal O_{Y^*}$ by \cite[Corollaire (4.2.8) (iii) b)]{deligne}. 
This is because $\Gr^W_l\!\left((\mathcal V^d_{Y^*})^*\right)$ is a polarizable 
variation of $\mathbb Q$-Hodge structure. 
Thus we get $\mathcal O_Y(aM_Y)\simeq 
\mathcal O_Y$ by taking the canonical extension. 
This is what we wanted. 
\end{step} 
Hence, we obtain $\mathbf M_{Y'}\sim _{\mathbb Q}0$. 
\end{proof}

We close this section with the proof of 
Corollary \ref{B-cor1.4}. 

\begin{proof}[Proof of Corollary \ref{B-cor1.4}] 
By \cite[Lemma 4.12]{fujino-slc-trivial}, 
we may assume that $Y$ is a smooth projective 
curve. We always have $\deg M_Y\geq 0$ since $M_Y$ 
is nef by \cite[Theorem 1.2]{fujino-slc-trivial}. 
If $\deg M_Y>0$, 
then it is obvious that 
$M_Y$ is ample. 
If $\deg M_Y=0$, then 
$M_Y$ is numerically trivial. 
In this case, by Theorem \ref{B-thm1.3}, 
$M_Y\sim _{\mathbb Q}0$ holds. 
Therefore, we see that $M_Y$ is always 
semi-ample. 
\end{proof}


\end{document}